\documentclass{amsart}
\usepackage[margin=2.8cm]{geometry}
\usepackage{graphics, amsmath,amssymb, enumitem, comment, url, mathtools}
\usepackage{graphicx}
\usepackage{epsfig}

\newif\ifcolorcomments

\usepackage{mathrsfs}  
\colorcommentstrue
\usepackage{amssymb}
\usepackage{amsmath,amsthm}

\usepackage{colortbl}

\newtheorem{theorem}{Theorem}[section]
\newtheorem{lemma}[theorem]{Lemma}
\newtheorem{proposition}[theorem]{Proposition}
\newtheorem{corollary}[theorem]{Corollary}
\theoremstyle{definition}

\newtheorem{remark}[theorem]{Remark}

\newtheorem{conjecture}[theorem]{Conjecture}
\newtheorem{problem}[theorem]{Problem}

\newcommand{\N}{\mathbb N}

\renewcommand{\text}{\textup}

\newcommand{\NPC}[1]{\ignorespaces}

\newif\ifdraft\drafttrue

\def\N{\mathbb N}

\newcommand\hdim{\dim_{\mathrm H}}

\IfFileExists{marvosym.sty}{
\RequirePackage{marvosym} 
}{

}

\IfFileExists{wasysym.sty}{
\RequirePackage{wasysym}\renewcommand{\emptyset}{{\diameter}}
}{

}

\mathtoolsset{showonlyrefs}
\newcommand*{\myDots}{\ifmmode\mathellipsis\else.\kern-0.07em.\kern-0.07em.\fi}

\usepackage{amsmath}

\newcommand {\ignore}[1] {}

\usepackage{tikz}

\begin{document}

\title{On a Conjecture of Cusick on a sum of Cantor sets}


 \author{Nikita Shulga}
 \address{Nikita Shulga,  Department of Mathematical and Physical Sciences,  La Trobe University, Bendigo 3552, Australia. }
 \email{n.shulga@latrobe.edu.au}
 \date{}

\maketitle

\begin{abstract}
In 1971 Cusick proved that every real number $x\in[0,1]$ can be expressed as a sum of two continued fractions with no partial quotients equal to $1$. 
In other words, if we define a set
$$
S(k):= \{ x\in[0,1] : a_n(x) \geq k \text{ for all } n\in\N \},
$$
then
$$
S(2)+S(2) = [0,1].
$$
He also conjectured that this result is unique in the sense that if you exclude partial quotients from $1$ to $k-1$ with $k\geq3$, then the Lebesgue measure $\lambda$ of the set of numbers which can be expressed as a sum of two continued fractions with no partial quotients from $\{1,\ldots,k-1\}$ is equal to $0$, that is
$$\lambda\Bigl( S(k)+S(k) \Bigl)= 0 \text{ for }k\geq 3.$$
In this paper, we disprove the conjecture of Cusick by showing that
$$
S(k)+S(k) \supseteq  \left[0,\frac{1}{k-1}\right].
$$
The proof is constructive and does not rely on ideas from previous works on the topic. We also show the existence of countably many 'gaps' in $S(k)+S(k)$, that is intervals, for which the endpoints lie in $S(k)+S(k)$, while none of the elements in the interior do so. Finally, we prove several results on the sums 
$$
S(m)+S(n)
$$
for $m\neq n$.
\end{abstract}

\section{Introduction and main results}

In 1947, Hall in his classical work \cite{MR0022568} proved that any real number $x$ modulo $1$ can be expressed as a sum of two continued fractions with partial quotients bounded above by $4$. 
Formally, define a set
$$
F(k) := \left\{ x\in[0,1]:  1\leq a_n(x) \leq k \text{ for all } n\geq1          \right\},
$$
where $a_n(x)$ is the $n$th partial quotient of regular continued fraction of $x$. Hall's result states that
$
F(4)+F(4)
$
contains an interval of length greater than $1$, thus it includes every real number modulo $1$. This also means that every real number $x$ can be expressed as a sum of two elements from $F(4)$ and an integer.

Furthermore, Cusick \cite{MR0309875} and Diviš \cite{MR0371826} independently proved that Hall's result is in some sense best possible. They showed that 
$ F(3)+F(3)$ does not contain an interval of length $\geq1$.

Hall's work inspired many generalizations and improvements of this result. In particular, Cusick posed a similar question about partial quotients bounded from below. To be specific, he defined a set
$$
S(k) := \left\{ x\in[0,1]:    a_n(x) \geq k \text{ for all } n\geq1          \right\}.
$$

For any rational number $p/q$, partial quotients $a_n(p/q)$ are undefined for large enough $n$, but we still include them in $S(k)$ if all partial quotients which are defined satisfy $a_n(p/q)\geq k$. We also include $0$ in $S(k)$, as one can think of $0$ as having a continued fraction $0=\frac{1}{\infty}$, that is it has one partial quotient $a_1(0)$ equals $+\infty$, which satisfies $a_1(0)\geq k$ for any $k$.

Using methods of Hall, Cusick \cite{MR0269603} proved that $S(2)+S(2)=[0,1]$, meaning that each real number $x\in [0,1]$ can be expressed as a sum of two continued fractions with no partial quotients equal to $1$.

In \cite{MR0282924} Cusick and Lee proved that $kS(k)=[0,1]$ so that every number in the interval $[0,1]$ is representable as a sum of $k$ elements of $S(k)$. This result can also be deduced from a later general theorem due to Hlavka \cite{MR0376545}.

Easy to see that for $k\geq3$, the sum $S(k)+S(k)$ cannot be equal to the whole interval $[0,1]$, as we trivially have $S(k)+S(k)\subseteq \left[0,\frac{2}{k} \right]$. For a fixed $k$ we call an interval $[0,\frac{2}{k}]$ a \textit{maximal} interval. However, one can show that in general $S(k)+S(k)\neq \left[0,\frac{2}{k} \right]$ as, for instance, an interval $\left(\frac{7}{12}, \frac{3}{5}\right) \subset \left[0,\frac{2}{3}\right]$ is not contained in $S(3)+S(3)$. Also, Cusick noted that on each level of building $S(k)$ as a Cantor set, we are removing a larger portion of each interval, which makes the original methods of Hall not very useful when applied to the set $S(k)$ with $k\geq3$. All this suggested the following conjecture, that Cusick formulated in \cite{MR0269603}.
\begin{conjecture}\label{conjecture:main}
    The set of real numbers which can be expressed as a sum of two copies of $S(k)$ for $k\geq3$ is of Lebesgue measure zero, that is
 $$\lambda\Bigl( S(k)+S(k) \Bigl)= 0 \text{ for }k\geq 3.$$
\end{conjecture}

 In a great paper \cite{MR1491854}, Astels provided several general theorems about the size of arithmetic sums and products of Cantor sets, which were applied to continued fractions sets. His main result is formulated in terms of {\it thickness} of a Cantor set and is applicable to a sum of $k$, possibly different, Cantor sets. The statement requires some definitions and notation, thus we skip it here and refer to \cite{MR1491854} for more details.

 Using this theorem, he recovered results by Hall, Cusick, Cusick-Lee, and some results by Hlavka mentioned below. However, once applied to the sumset $S(k)+S(k)$, this theorem only provides a lower bound on the Hausdorff dimension, which for a set $X$ is denoted by $\hdim X$. To be more precise, the current best-known result about the size of $S(k)+S(k)$, which follows from Astels' result, is the bound
\begin{equation}\label{hdim:bound:astels}
\hdim \left( S(k)+S(k) \right) \geq \frac{\log 2}{\log\left( 1+\frac{k}{2} \right)}.
\end{equation}
 
In particular, for $k=3$ and $k=6$ this bound becomes
$$
\hdim \left( S(3)+S(3) \right) \geq \frac{\log 2}{\log\left( \frac{5}{2} \right)}=0.756^+ \text{ and }  \hdim \left( S(6)+S(6) \right) \geq \frac{1}{2}.
$$

The (sufficient) condition provided by Astels for a sum of Cantor sets to contain an interval fails for $S(k)+S(k)$ for any $k>2$, which also suggested that for $k\geq3$ this sumset is of zero Lebesgue measure. We also note that the bound \eqref{hdim:bound:astels} becomes trivial for all $k\geq6$ due to the result of Good \cite{MR0004878}, which states
$$
\hdim S(k) \geq \frac{1}{2}+\frac{1}{2\log(k+2)} \text{ for } k\geq14.
$$

The current best-known asymptotic formula for $\hdim S(k)$ was recently obtained in \cite{MR4540838} using thermodynamic formalism and functional analysis. See also \cite{MR2672614} for a slightly worse asymptotic result.

\vskip0.3cm

In this manuscript, we disprove Conjecture \ref{conjecture:main} on a sum of two Cantor sets. Namely, we prove the following result.
\begin{theorem}\label{theorem:main}
For any $k\geq2$, the sum $S(k)+S(k)$ contains an interval $\left[0,\frac{1}{k-1}\right]$. 

Thus, 
 $$
 \lambda\Bigl( S(k)+S(k) \Bigl)\geq \frac{1}{k-1} \text{ for }k\geq 2.
 $$
\end{theorem}
Note that for $k=2$ we recover the main result of \cite{MR0269603}.

We noticed above that $S(k) + S(k) \subseteq \left[0,\frac{2}{k}\right]$. 
As the conjecture about zero measure turned out to be false, one may conjecture the natural alternative that the sum is equal to the maximal interval at least for large enough $k$, that is
$$
S(k) + S(k) = \left[0,\frac{2}{k}\right].
$$
However, in our next result we show that this is false. We will provide a countable collection of intervals, such that endpoints of each of them belong to $S(k)+S(k)$, while none of the points in the interior lie in $S(k)+S(k)$.


For a fixed $k\geq1$ and any $n\geq1$, let $M_{k,n} = [\underbrace{k,\ldots,k}_{n \text{ times }}]$ and $m_{k,n} = [\underbrace{k,\ldots,k}_{n \text{ times }},1]$. Next, for a fixed $k\geq3$ and for any $n\geq1$, define an open interval 
$$
G_{k,n} = \begin{cases}
  \left( M_{k,n}+m_{k,n}, 2M_{k,n+1} \right) \quad \text{if $n$ is odd,}\\
    \left(2M_{k,n+1}, M_{k,n}+m_{k,n} \right) \quad \text{if $n$ is even.} 
\end{cases}
$$
One can check that the $G_{k,n}$ as an interval is defined correctly, that is the right-most point is larger than the left-most point for all $k\geq3$ and $n\geq1$.


We prove the following
\begin{theorem}\label{prop:gaps}
    For any $k\geq3$ we have
$$
S(k)+S(k) \subseteq \left[0,\frac2k\right] \setminus \bigcup_{n=1}^\infty G_{k,n}.
$$

\end{theorem}

For a final set of results, we first mention that for the sets $F(k)$ several authors (see, for example, \cite{MR1876272,MR1866013,MR0376545}) have studied arithmetic sums of two different Cantor sets
$$
F(m)+F(n).
$$
In 1975, Hlavka \cite{MR0376545} examined these sums and proved that
$$
\lambda \left(F(m)+F(n)\right)\geq1
$$
holds for $(m, n)$ equal to $(2, 7)$ or $(3, 4)$, but does not hold for $(m, n)$ equal
to $(2, 4)$. 

We consider an analogous question for the sums of two different Cantor sets $S(m)$ and $S(n)$. As the sets $S(k)$ are nested, that is
$$
[0,1]=S(1)\supset S(2)\supset\ldots\supset S(k)\supset S(k+1)\supset\ldots,
$$
we see that Theorem \ref{theorem:main} immediately implies the following result about $S(m)+S(n)$.
\begin{corollary}
    For any integers $m,n\geq2$, we have
    $$
 \lambda\Bigl( S(m)+S(n) \Bigl)\geq \frac{1}{\max(m,n)-1}.
 $$
\end{corollary}

As it will be seen, our method also allows us to improve this 'trivial' bound under some conditions on $m,n$.
\begin{theorem}\label{main:theorem2}
    If $m,n$ are integers, such that 
$3 \leq m < n\leq (m-1)^2$, then
$$ 
S(m)+S(n) \supseteq  \left[0, \frac{1}{m-1}\right].
$$
\end{theorem}



For some particular values of $m$ and $n$, we can get an optimal statement, showing that the sum $S(m)+S(n)$ is exactly equal to some interval. We provide one example of such result.
\begin{theorem}\label{thm:square}
    For any integer $m\geq2$, one has
    $$
S(m)+S(m^2) = \left[ 0, \frac{m+1}{m^2} \right].
$$
\end{theorem}

Finally, we mention that the topic of sums and products of Cantor sets, not necessarily formulated in terms of continued fractions, is also well-studied, see, for example, \cite{MR2534296, MR4405675, MR1267692, MR3639301,MR3976578} and references therein. 

The paper is organized as follows. In Section \ref{sec:prelim} we provide some basic facts from continued fraction theory and the necessary notation. In Section \ref{sec:main:construction} we describe the algorithm of expressing a given number as a sum of two continued fractions with "large" partial quotients and prove several properties of it. In Section \ref{sec:proofs} we deduce Theorems \ref{theorem:main}, \ref{main:theorem2}, \ref{thm:square} from this algorithmic procedure. In Section \ref{sec:gaps} we prove Theorem \ref{prop:gaps} about the gaps in $S(k)+S(k)$. In Section \ref{sec:remarks} we give some concluding remarks and formulate open problems. 

\section{Preliminaries}\label{sec:prelim}

For $x\in [0,1]$ we consider its continued fraction expansion
$$x=[a_1,a_2,\ldots,a_n,\ldots]=  \cfrac{1}{a_1+\cfrac{1}{a_2+\cdots}}, \,\,\, a_j \in \mathbb{Z}_+$$
which is unique and infinite when $x\not\in\mathbb{Q}$  and finite for rational $x$. 
Each rational $x$ has just two representations
$$
x=[a_1,a_2,\ldots,a_{n-1}, a_n]  
\,\,\,\,\,\text{and}\,\,\,\,\,
x=[a_1,a_2,\ldots,a_{n-1}, a_n-1,1]
,
\,\,\,\, \text{where}\,\,\,\, a_n \ge 2
.
$$
Going forward, when talking about rational numbers, we always think of them as having the last partial quotient to be at least $2$. We use the notation $a_n(x)$ for a function which returns $n$th partial quotient of a number $x$.

Also, for any $n\in\N$ and any $\left(c_{1},\ldots,c_{n}\right)$ where $c_{i}\in\mathbb{N}$ for $\, 1\leq i \leq n$, we define $\textit{a cylinder of order }n$, denoted by $I_{n}\left(c_{1},\ldots,c_{n}\right)$, as
\begin{align*}
I_{n}\left(c_{1},\ldots,c_{n}\right):=\left\{x\in\left[0,1\right):a_{1}\left(x\right)=c_{1},\ldots,a_{n}\left(x\right)=c_{n}\right\}.
\end{align*}
We would use the following basic facts about cylinders.
\begin{proposition}\label{range}
For any $n\ge 1$ and $\left(c_{1},\ldots,c_{n}\right)\in\mathbb{N}^{n}$, we have
\begin{equation}I_{n}\left(c_{1},\ldots,c_{n}\right)=\left\{
  \begin{array}{ll}
   \left[\dfrac{p_{n}}{q_{n}},\dfrac{p_{n}+p_{n-1}}{q_{n}+q_{n-1}}\right) &\textrm{if }n\textrm{ is even},\\
   \left(\dfrac{p_{n}+p_{n-1}}{q_{n}+q_{n-1}},\dfrac{p_{n}}{q_{n}}\right] &\textrm{if }n\textrm{ is odd}.\\
  \end{array}
\right.
\end{equation}
Therefore, the length of it is given by
\begin{equation}\label{length}
\left|I_{n}\left(c_{1},\ldots,c_{n}\right)\right|=\dfrac{1}{q_{n}\left(q_{n}+q_{n-1}\right)}.
\end{equation}
\end{proposition}
The proof can be found in classic continued fractions books, see, for example, \cite{Khinchin_book}.

We also use the notation $\langle c_1,\ldots,c_n \rangle$ to denote a {\it continuant}, that is the denominator $q_n:=q_n(c_1,\ldots,c_n)$ of a continued fraction $[c_1,\ldots,c_n]$.

Using this notation, we recall the classical formula
$$
\langle c_1,\ldots,c_n \rangle = \langle c_1,\ldots,c_k \rangle \langle c_{k+1},\ldots,c_n \rangle  + \langle c_1,\ldots,c_{k-1} \rangle \langle c_{k+2},\ldots,c_n \rangle
$$
for $1\leq k \leq n-1$. For $k=n-1$ we use the convention $\langle c_{k+2},\ldots,c_n \rangle=1,$ as $c_{n+1}$ is undefined.

The following statement allows us to bound the ratio of two continuants of the same length under the condition that partial quotients in one are greater or equal to the partial quotients in the other one.



\begin{lemma}\label{lemma:quotient:tnqn}
Consider two integer sequences $\{c_n\}_{n\geq1}$ and $\{b_n\}_{n\geq1}$, satisfying $b_n \geq c_n$ for all $n$.

Then for any $n\geq2$ and any $1\leq k\leq n-1$ one has
$$
\frac{\langle b_1,\ldots,b_n\rangle}{\langle c_1,\ldots,c_n\rangle} \geq 1 + \frac{\langle b_1,\ldots, b_k\rangle - \langle c_1,\ldots,c_k\rangle + (\langle b_1,\ldots, b_{k-1}\rangle - \langle c_1,\ldots,c_{k-1}\rangle ) \frac{1}{c_{k+1}+1}}{\langle c_1,\ldots,c_k\rangle +\langle c_1,\ldots,c_{k-1} \rangle \frac{1}{c_{k+1}}}.
$$
\end{lemma}

\begin{proof}
Continuant is a strictly increasing function of its arguments, so we have
\begin{align}
\frac{\langle b_1,\ldots,b_n\rangle}{\langle c_1,\ldots,c_n\rangle} &\geq \frac{\langle b_1,\ldots,b_k,c_{k+1},\ldots,c_n\rangle}{\langle c_1,\ldots,c_k,c_{k+1},\ldots,c_n\rangle} \\
&= \frac{\langle b_1,\ldots, b_k\rangle \langle c_{k+1},\ldots,c_n\rangle + \langle b_1,\ldots, b_{k-1}\rangle \langle c_{k+2},\ldots,c_n\rangle }{\langle c_1,\ldots, c_k\rangle \langle c_{k+1},\ldots,c_n\rangle + \langle c_1,\ldots, c_{k-1}\rangle \langle c_{k+2},\ldots,c_n\rangle} \\
& =1 + \frac{(\langle b_1,\ldots, b_k\rangle - \langle c_1,\ldots, c_k\rangle) \langle c_{k+1},\ldots,c_n\rangle + (\langle b_1,\ldots, b_{k-1}\rangle-\langle c_1,\ldots, c_{k-1}\rangle) \langle c_{k+2},\ldots,c_n\rangle }{\langle c_1,\ldots, c_k\rangle \langle c_{k+1},\ldots,c_n\rangle + \langle c_1,\ldots, c_{k-1}\rangle \langle c_{k+2},\ldots,c_n\rangle} \\
& = 1 + \frac{\langle b_1,\ldots, b_k\rangle - \langle c_1,\ldots, c_k\rangle  + (\langle b_1,\ldots, b_{k-1}\rangle-\langle c_1,\ldots, c_{k-1}\rangle) \frac{\langle c_{k+2},\ldots,c_n\rangle}{\langle c_{k+1},\ldots,c_n\rangle} }{\langle c_1,\ldots, c_k\rangle  + \langle c_1,\ldots, c_{k-1}\rangle \frac{ \langle c_{k+2},\ldots,c_n\rangle}{\langle c_{k+1},\ldots,c_n\rangle}}  \\
& \geq 1 + \frac{\langle b_1,\ldots, b_k\rangle - \langle c_1,\ldots,c_k\rangle + (\langle b_1,\ldots, b_{k-1}\rangle - \langle c_1,\ldots,c_{k-1}\rangle ) \frac{1}{c_{k+1}+1}}{\langle c_1,\ldots,c_k\rangle +\langle c_1,\ldots,c_{k-1} \rangle \frac{1}{c_{k+1}}} .
\end{align}
\end{proof}

In particular, by letting $k=1$ in Lemma \ref{lemma:quotient:tnqn} we get
$$
\frac{\langle b_1,\ldots,b_n\rangle}{\langle c_1,\ldots,c_n\rangle} \geq 1 + \frac{b_1-c_1}{c_1+1/c_2}.
$$


\section{The main construction}\label{sec:main:construction}

We will explicitly provide an interval and an algorithm, by which you can decompose any number from this interval into a sum of two continued fractions with no small partial quotients.

First, let us introduce relevant definitions and notation.
For infinite (or finite) continued fractions 
$$c=[c_1,\ldots,c_n,\ldots] \quad \text{ and }\quad
 b=[b_1,\ldots,b_n,\ldots],$$
denote their finite truncations by 
\begin{equation}\label{def:pq:st}
\frac{p_n}{q_n} = [c_1,\ldots,c_n] \quad \text{ and } \quad \frac{s_n}{t_n} = [b_1,\ldots,b_n]
\end{equation}
respectively.




The number $x=0$ is trivially an element of $S(k)+S(k)$ for any $k\geq1$. For a given $x\in(0,1]$ we want to construct two sequences $\frac{p_n}{q_n}=[c_1,\ldots,c_n]$ and $\frac{s_n}{t_n}=[b_1,\ldots,b_n]$ of rational numbers, such that 
$$x=\lim_{n\to\infty} \left( \frac{p_n}{q_n} + \frac{s_n}{t_n} \right)$$
with the property that $p_n/q_n$ is a good approximation to $x- s_{n-1}/t_{n-1}$, and simultaneously $s_n/t_n$ is a good approximation to $x-p_n/q_n$.




Set $c_1  = a_1(x)+1$ and for $n\geq1$ iteratively define
\begin{equation}\label{algorithm:bn}
b_n = a_n \left(    x - \frac{p_n}{q_n}  \right)
\end{equation}
and
\begin{equation}\label{algorithm:cn}
c_{n+1} = a_{n+1} \left(    x - \frac{s_n}{t_n}  \right)  +1.
\end{equation}

\begin{remark}\label{remark:rational:termination}
For rational numbers $x$ for some $n\in\N$ the step  \eqref{algorithm:cn} can be undefined. It can happen if $x-s_n/t_n$ has strictly less than $n+1$ partial quotients. In this case, we terminate the algorithm at the step this situation occurs.

For example, let $x=1$. In this case $c_1=a_1(1)+1 = 2$ and $b_1 = a_1(1/2) = 2$. Then $c_2 = a_2(1/2)$, which is undefined, as the number $1/2=[2]$ has only one partial quotient (recall that we fixed continued expansions of rationals to have the last partial quotient equal at least $2$). 
\end{remark}

For the algorithm defined by \eqref{algorithm:bn}, \eqref{algorithm:cn} we claim that 
\begin{equation}\label{eq:convergence}
\lim_{n\to\infty} \left( \frac{p_n}{q_n} + \frac{s_n}{t_n} \right) = x
\end{equation}
with $c_i\geq k$ and $b_i\geq k$ for all $i\in\N$.



We need to check the following three properties.
\begin{itemize}
    \item  Correctness of the definition of the algorithm;

    \item Convergence to the number $x$;
    
    \item 'Large' size of the resulting partial quotients.
\end{itemize}

\begin{remark}
    If the algorithm terminates in a finite number of steps as per Remark \ref{remark:rational:termination}, then we drop the limit from \eqref{eq:convergence} and claim that $c_i\geq k$ and $b_i\geq k$ for all $1\leq i \leq n$.

It will be evident that in this case all three properties persist and we would just get a decomposition of $x$ into a sum of two rational numbers $b$ and $c$.
\end{remark}

First, let us check that the algorithm is defined correctly, that is we need to check that a number
$
x - p_n/q_n,
$
which defines a partial quotient $b_n$, satisfies
\begin{equation}\label{condition:cylinder:x-pq}
a_i \left( x - \frac{p_n}{q_n} \right) = b_i  \text{ for } 1\leq i \leq n-1
\end{equation}
and that a number 
$
x-s_n/t_n,
$
which defines a partial quotient $c_{n+1}$, satisfies
\begin{equation}\label{condition:cylinder:x-st}
a_i \left( x - \frac{s_n}{t_n} \right) = c_i  \text{ for } 1\leq i \leq n.
\end{equation}

We shall prove this by induction. The first partial quotient $c_1$ is clearly correctly defined, this is a unique integer number, satisfying inequality
\begin{equation}\label{def:c1}
\frac{1}{c_1} < x \leq \frac{1}{c_1-1}.
\end{equation}
The first partial quotient $b_1$ of a number $b$ is also correctly defined, this is a unique integer satisfying inequality
\begin{equation}\label{def:b1}
\frac{1}{b_1+1} < x-\frac{1}{c_1} \leq \frac{1}{b_1}.
\end{equation}
It exists because $x-\frac{1}{c_1}$ is a positive number. If $x-\frac{1}{c_1}=\frac{1}{b_1}$, then we terminate the algorithm at this step. Otherwise, we proceed further.

Next, by our algorithm, $c_2$ is defined as a solution to the inequality
$$
\frac{p_2-p_1}{q_2-q_1} \leq x-\frac{1}{b_1} < \frac{p_2}{q_2}.
$$
We want to show that $c_2$, satisfying the inequality above, always exists.

For $c_2$ to be correctly defined, we want
$$
x-\frac{1}{b_1} \in I_1(c_1).
$$
This is equivalent to 
$$
\frac{1}{c_1+1} < x-\frac{1}{b_1} \leq  \frac{1}{c_1}.
$$
The inequality \eqref{def:b1} implies that the right-hand side of the latter inequality is trivially satisfied while showing the left-hand side is enough to prove that
$$
\frac{1}{c_1+1}  \leq \frac{1}{c_1} - \frac{1}{b_1(b_1+1)}.
$$
To satisfy this, it is enough to prove that $b_1\geq c_1.$
From inequalities \eqref{def:c1} and \eqref{def:b1} we easily get
$$
b_1+1 > c_1(c_1-1),
$$
or, as both sides of this inequality are integer,
$$
b_1 \geq c_1(c_1-1) \geq c_1.
$$
The last inequality is because $c_1\geq2$. Thus $c_2$ is always correctly defined.

Now, assume that $c_1,b_1,\ldots,b_{k-1},c_{k}$ are correctly defined. We want to show that $b_{k}$ and $c_{k+1}$ would also be correctly defined. First, let us prove the induction step for $b_k$.

\begin{lemma}\label{bk_correctly_defined}
    Assume that $c_1,b_1,\ldots,b_{k-1},c_k$ are correctly defined. Then $b_k$ is also correctly defined.
\end{lemma}

\begin{proof}
Case 1. $k=2n$ is an even number. 

Inequality which defines $b_{2n}$ is 
\begin{equation}\label{deff:b2n}
\frac{s_{2n}}{t_{2n}} \leq x - \frac{p_{2n}}{q_{2n}} < \frac{s_{2n}+s_{2n-1}}{t_{2n}+t_{2n-1}}.
\end{equation}

The condition \eqref{condition:cylinder:x-pq} is equivalent to the fact that 
$$
x-\frac{p_{2n}}{q_{2n}} \in I_{2n-1}(b_1,\ldots,b_{2n-1}),
$$
which in turn by Proposition \ref{range} is equivalent to the inequality
\begin{equation}\label{induction:b2n}
\frac{s_{2n-1}+s_{2n-2}}{t_{2n-1}+t_{2n-2}} < x - \frac{p_{2n}}{q_{2n}} \leq \frac{s_{2n-1}}{t_{2n-1}}.
\end{equation}
We want to show that a number $x-p_{2n}/q_{2n}$ will always satisfy the inequality \eqref{induction:b2n}.

By inductive hypothesis, we know that
$$
\frac{p_{2n}-p_{2n-1}}{q_{2n}-q_{2n-1}} \leq x - \frac{s_{2n-1}}{t_{2n-1}} < \frac{p_{2n}}{q_{2n}},
$$
which can be rewritten as 
$$
\frac{s_{2n-1}}{t_{2n-1}} - \frac{1}{q_{2n}(q_{2n}-q_{2n-1})} \leq x - \frac{p_{2n}}{q_{2n}} < \frac{s_{2n-1}}{t_{2n-1}}.
$$
Thus right-hand side of inequality \eqref{induction:b2n} is automatically true. Now we need to deal with the left-hand side.

It is enough to show that
$$
\frac{s_{2n-1}}{t_{2n-1}} - \frac{1}{q_{2n}(q_{2n}-q_{2n-1})} > \frac{s_{2n-1}+s_{2n-2}}{t_{2n-1}+t_{2n-2}}
$$
or, rewritten using properties of neighbor Farey fractions,
\begin{equation}
\frac{1}{t_{2n-1}(t_{2n-1}+t_{2n-2})} > \frac{1}{q_{2n}(q_{2n}-q_{2n-1})},
\end{equation}
which is just
\begin{equation}\label{induction:b2n:final}
q_{2n}(q_{2n}-q_{2n-1}) >t_{2n-1}(t_{2n-1}+t_{2n-2}).
\end{equation}

Case 2.  $k=2n+1$ is an odd number. 

Inequality which defines $b_{2n+1}$ is 
\begin{equation}\label{deff:b2n+1}
\frac{s_{2n+1}+s_{2n}}{t_{2n+1}+t_{2n}} < x - \frac{p_{2n+1}}{q_{2n+1}} \leq    \frac{s_{2n+1}}{t_{2n+1}}.
\end{equation}

The condition \eqref{condition:cylinder:x-pq} is equivalent to the fact that 
$$
x-\frac{p_{2n+1}}{q_{2n+1}} \in I_{2n}(b_1,\ldots,b_{2n}),
$$
which in turn by Proposition \ref{range} is equivalent to the inequality
\begin{equation}\label{induction:b2n+1}
\frac{s_{2n}}{t_{2n}}   \leq x - \frac{p_{2n+1}}{q_{2n+1}} < \frac{s_{2n}+s_{2n-1}}{t_{2n}+t_{2n-1}}.
\end{equation}
We want to show that a number $x-p_{2n+1}/q_{2n+1}$ will always satisfy the inequality \eqref{induction:b2n+1}.

By inductive hypothesis, we know that
$$
\frac{p_{2n+1}}{q_{2n+1}}  < x - \frac{s_{2n}}{t_{2n}} \leq \frac{p_{2n+1}-p_{2n}}{q_{2n+1}-q_{2n}},
$$
which can be rewritten as 
$$
\frac{s_{2n}}{t_{2n}}  < x - \frac{p_{2n+1}}{q_{2n+1}} \leq \frac{s_{2n}}{t_{2n}}+ \frac{1}{q_{2n+1}(q_{2n+1}-q_{2n})}.
$$
Thus left-hand side of inequality \eqref{induction:b2n+1} is automatically true. Now we need to deal with the right-hand side.

It is enough to show that
$$
\frac{s_{2n}}{t_{2n}}+ \frac{1}{q_{2n+1}(q_{2n+1}-q_{2n})} < \frac{s_{2n}+s_{2n-1}}{t_{2n}+t_{2n-1}},
$$
or, rewritten using properties of neighbor Farey fractions,
\begin{equation}
\frac{1}{t_{2n}(t_{2n}+t_{2n-1})} > \frac{1}{q_{2n+1}(q_{2n+1}-q_{2n})},
\end{equation}
which is just
\begin{equation}\label{induction:b2n+1:final}
q_{2n+1}(q_{2n+1}-q_{2n}) >t_{2n}(t_{2n}+t_{2n-1}).
\end{equation}

As one can see, both \eqref{induction:b2n:final} and \eqref{induction:b2n+1:final} can be written in terms of $k$ as
\begin{equation}\label{induction:bnbn:finalefinale}
q_k(q_k-q_{k-1}) > t_{k-1}(t_{k-1}+t_{k-2}).
\end{equation}

To show this we need the following lemma.
 
\begin{lemma}\label{lemma:quotient:qlarge}
Assume that partial quotients $c_1,b_1,\ldots,b_{k-1},c_k$ are correctly defined. Then for all $1\leq n \leq k$ we have 
\begin{equation}\label{inequality:lemma1}
q_{n}q_{n-1} > t_{n-1}(t_{n-1}+t_{n-2}) .
\end{equation}
\end{lemma}

\begin{proof}
We will show this for the odd index $2m+1\leq k$. For even indices, the process is exactly the same except for inequality signs due to the properties of even and odd convergents. By assumption, all partial quotients $c_i$ with indices smaller or equal to $2m+1$ and all partial quotients $b_i$ with indices smaller or equal to $2k$ are correctly defined.

Thus by our procedure, $c_{2m+1}$ satisfies
$$
\frac{p_{2m+1}}{q_{2m+1}}  < x - \frac{s_{2m}}{t_{2m}}   \leq \frac{p_{2m+1}-p_{2m}}{q_{2m+1}-q_{2m}}.
$$
This implies
\begin{equation}\label{eq:good:first}
\frac{1}{q_{2m}q_{2m+1}}  < \left| x - \frac{s_{2m}}{t_{2m}} - \frac{p_{2m}}{q_{2m}} \right| \leq \frac{1}{q_{2m}(q_{2m+1}-q_{2m})}.
\end{equation}
On the other hand, by procedure we also have
$$
\frac{s_{2m}}{t_{2m}} \leq x- \frac{p_{2m}}{q_{2m}} < \frac{s_{2m}+s_{2m-1}}{t_{2m}+t_{2m-1}}.
$$
Thus
\begin{equation}\label{eq:good:second}
0 \leq \left| x - \frac{s_{2m}}{t_{2m}} - \frac{p_{2m}}{q_{2m}} \right| < \frac{1}{t_{2m}(t_{2m}+t_{2m-1})}.
\end{equation}
Combining left-hand side of \eqref{eq:good:first} with right-hand side of \eqref{eq:good:second} we get
$$
q_{2m}q_{2m+1} > t_{2m}(t_{2m}+t_{2m-1}).
$$
\end{proof}

\noindent
\textit{End of proof of Lemma \ref{bk_correctly_defined}. }
By Lemma \ref{lemma:quotient:qlarge} for $n=k$, we get
$$
q_{k}q_{k-1} > t_{k-1} (t_{k-1}+t_{k-2})
$$
Thus in order to secure \eqref{induction:bnbn:finalefinale} we need to check that
$$
q_{k}q_{k-1} < q_k(q_k-q_{k-1})
$$
The latter is equivalent to 
$$
2 q_{k-1} < q_{k}=c_{k} q_{k-1} + q_{k-2},
$$
which is always true as $c_i\geq2$ for all $i\in\N$ just by definition of procedure \eqref{algorithm:cn}.

Thus $b_k$ can always be correctly defined assuming that all previous partial quotients are correctly defined. 
\end{proof}

\begin{remark}\label{remark:termination:exp1}
In Remark \ref{remark:rational:termination} we described a situation when the algorithm terminates in finite numbers of steps. This corresponds to the case when $x-p_{2n}/q_{2n}=s_{2n}/t_{2n}$ in \eqref{deff:b2n} or $x-p_{2n+1}/q_{2n+1}=s_{2n+1}/t_{2n+1}$ in \eqref{deff:b2n+1}, both of which are legitimate solutions to the inequalities. If this happens, then we get a finite decomposition, and we do not need to proceed further.
\end{remark}

Now we need to show the induction step for partial quotients of the limit number $c$ assuming that the situation described in Remark \ref{remark:termination:exp1} did not occur.

\begin{lemma}\label{ck+1_correctly_defined}
    Assume that $c_1,b_1,\ldots,c_k,b_k$ are correctly defined and that the algorithm did not terminate yet. Then $c_{k+1}$ is also correctly defined.
\end{lemma}
\begin{proof}
Once again, we will consider only the case of odd index $k+1=2n+1$.

Inequality which defines $c_{2n+1}$ is 
$$
\frac{p_{2n+1}}{q_{2n+1}} < x - \frac{s_{2n}}{t_{2n}} \leq \frac{p_{2n+1}-p_{2n}}{q_{2n+1}-q_{2n}}.
$$

The condition \eqref{condition:cylinder:x-st} is equivalent to the fact that 
$$
x-\frac{s_{2n}}{t_{2n}} \in I_{2n}(c_1,\ldots,c_{2n}),
$$
which in turn by Proposition \ref{range} is equivalent to the inequality
\begin{equation}\label{induction:c2n+1}
\frac{p_{2n}}{q_{2n}} \leq x - \frac{s_{2n}}{t_{2n}} <   \frac{p_{2n}+p_{2n-1}}{q_{2n}+q_{2n-1}}.
\end{equation}
We want to show that a number $x-s_{2n}/t_{2n}$ will always satisfy the inequality \eqref{induction:c2n+1}.

By the inductive hypothesis, we know 
$$
\frac{s_{2n}}{t_{2n}} \leq x - \frac{p_{2n}}{q_{2n}} < \frac{s_{2n}+s_{2n-1}}{t_{2n}+t_{2n-1}},
$$
which can be rewritten as 
$$
\frac{p_{2n}}{q_{2n}} \leq x - \frac{s_{2n}}{t_{2n}} <\frac{p_{2n}}{q_{2n}} + \frac{1}{t_{2n}(t_{2n}+t_{2n-1})}.
$$
Thus left-hand side of inequality \eqref{induction:c2n+1} is automatically true. Now we need to deal with the right-hand side.

It is enough to show that
$$
\frac{p_{2n}}{q_{2n}} + \frac{1}{t_{2n}(t_{2n}+t_{2n-1})} < \frac{p_{2n}+p_{2n-1}}{q_{2n}+q_{2n-1}},
$$
or, rewritten using properties of neighbor Farey fractions,
\begin{equation}
\frac{1}{t_{2n}(t_{2n}+t_{2n-1})} < \frac{1}{q_{2n}(q_{2n}+q_{2n-1})},
\end{equation}
which is just
\begin{equation}\label{induction:c2n+1:final}
t_{2n}(t_{2n}+t_{2n-1}) >q_{2n}(q_{2n}+q_{2n-1}).
\end{equation}

Easy to see that proving $b_n\geq c_n$ for any $n\in\N$ would be sufficient to show that \eqref{induction:c2n+1:final} is true. In order to show that $b_n\geq c_n$, we first formulate the following statement.

\begin{lemma}\label{lemma:quotient:tlarge}
Assume that partial quotients $c_1,b_1,\ldots,b_{k-1},c_k,b_k$ are correctly defined. Then for all $1\leq n \leq k$ we have 
\begin{equation}\label{inequality:lemma2}
t_{n-1}(t_n+t_{n-1}) > q_n(q_n - q_{n-1}).
\end{equation}
\end{lemma}
\begin{proof}
We will show this for the odd index $2m+1\leq n$. For even indices, the process is exactly the same except for inequality signs due to the properties of even and odd convergents. By assumption, all partial quotients $c_i$ with indices smaller or equal to $2m+1$ and all partial quotients $b_i$ with indices smaller or equal to $2m+1$ are correctly defined.

By our procedure, $b_{2m+1}$ satisfies
$$
\frac{s_{2m+1}+s_{2m}}{t_{2m+1}+t_{2m}}  < x - \frac{p_{2m+1}}{q_{2m+1}}   \leq \frac{s_{2m+1}}{t_{2m+1}}.
$$
This implies
\begin{equation}\label{eq:bad:first}
\frac{1}{t_{2m}(t_{2m+1}+t_{2m})}  < \left| x - \frac{s_{2m}}{t_{2m}} - \frac{p_{2m+1}}{q_{2m+1}} \right| \leq \frac{1}{t_{2m+1}t_{2m}}.
\end{equation}
On the other hand, by the procedure $c_{2m+1}$ was defined from 
$$
\frac{p_{2m+1}}{q_{2m+1}} < x- \frac{s_{2m}}{t_{2m}} \leq \frac{p_{2m+1}-p_{2m}}{q_{2m+1}-q_{2m}}.
$$
Thus
\begin{equation}\label{eq:bad:second}
0 < \left| x - \frac{s_{2m}}{t_{2m}} - \frac{p_{2m+1}}{q_{2m+1}} \right| \leq \frac{1}{q_{2m+1}(q_{2m+1}-q_{2m})}.
\end{equation}
Combining left-hand side of \eqref{eq:bad:first} with the right-hand side of \eqref{eq:bad:second} we get
$$
 t_{2m}(t_{2m+1}+t_{2m}) >q_{2m+1}(q_{2m+1}-q_{2m}) .
$$

\end{proof}



 Now our goal is to show that the decomposition of a number $x$ as a sum $x=c+b$, defined by algorithm \eqref{algorithm:bn} and \eqref{algorithm:cn} will generate only large partial quotients.


\begin{lemma}\label{lemma:growth:bn}
Assume that partial quotients $c_1,b_1,\ldots,b_{k},c_{k+1}$ are correctly defined. Then 
$$
b_{k+1}\geq c_{k+1}.
$$  
\end{lemma}
\begin{proof}
We will prove the statement when $k$ is even. When it is odd, all inequalities are similar and yield the same result.

When $k=2m$ is even, and as $b_{2m}$ is defined and $b_{2m+1}$ is always correctly defined by Lemma \ref{bk_correctly_defined},  we get
$$
\frac{s_{2m}}{t_{2m}} \leq  x - \frac{p_{2m}}{q_{2m}} < \frac{s_{2m}+s_{2m-1}}{t_{2m}+t_{2m-1}}
$$
and 
$$
\frac{s_{2m+1}+s_{2m}}{t_{2m+1}+t_{2m}}  < x - \frac{p_{2m+1}}{q_{2m+1}}   \leq \frac{s_{2m+1}}{t_{2m+1}}.
$$
As $\frac{p_{2m}}{q_{2m}}<\frac{p_{2m+1}}{q_{2m+1}} $, we conclude inequalities
$$
\frac{s_{2m+1}+s_{2m}}{t_{2m+1}+t_{2m}}\leq x - \frac{p_{2m+1}}{q_{2m+1}} < x - \frac{p_{2m}}{q_{2m}} \leq \frac{s_{2m}+s_{2m-1}}{t_{2m}+t_{2m-1}}.
$$
Thus
$$
\frac{1}{q_{2m}q_{2m+1}} \leq \frac{b_{2m+1}}{(t_{2m+1}+t_{2m})(t_{2m}+t_{2m-1})}.
$$
Using 
$$
t_{2m+1}+t_{2m} > \frac{q_{2m+1}(q_{2m+1}-q_{2m})}{t_{2m}}
$$
from Lemma \ref{lemma:quotient:tlarge}, we get
$$
b_{2m+1} > \frac{q_{2m+1}-q_{2m}}{q_{2m}}\cdot \frac{t_{2m}+t_{2m-1}}{t_{2m}} = \left( \underbrace{c_{2m+1}-1}_{\geq1}+\underbrace{\frac{q_{2m-1}}{q_{2m}}}_{>0} \right)\cdot \left( \underbrace{1+\frac{t_{2m-1}}{t_{2m}}}_{>1} \right) > c_{2m+1}-1.
$$
As $c_i,b_i$ are integers for any $i\in\N$, we conclude $b_{2m+1} \geq c_{2m+1}.$
\end{proof}

\textit{End of proof of Lemma \ref{ck+1_correctly_defined}. }
As we proved that $b_n\geq c_n$, we conclude that inequality \eqref{induction:c2n+1:final} is true, thus $c_{k+1}$ is always correctly defined assuming that all previous partial quotients are correctly defined. \end{proof}

    

Note that Lemma \ref{lemma:growth:bn} immediately implies the main result of \cite{MR0269603}.

\begin{corollary}\label{coroll:cusick}
     $S(2)+S(2)=[0,1]$.
     \end{corollary}
     \begin{proof}
     Indeed, by procedure \eqref{algorithm:cn} we trivially have $c_n\geq2$ for all $n\in\N$ and from Lemma \ref{lemma:growth:bn} we get $b_n\geq c_n\geq2$ for all $n\in\N$. As the algorithm is defined for any real $x\in(0,1]$, we get the result.
     \end{proof}





Now we can show that the sequence $\{c_n\}_{n\geq1}$ is bounded from below.
\begin{lemma}\label{lemma:growth:cn}
Let $c_1\geq3$. Then for the algorithm defined in \eqref{algorithm:bn}, \eqref{algorithm:cn} we have
\begin{equation}\label{ineq:cn:growing}
    c_{n+1} \geq (c_1 -1)^2 \text{ for any } n\geq1.
\end{equation}
\end{lemma}

\begin{proof}
First, note that the inequality \eqref{inequality:lemma1} from Lemma \ref{lemma:quotient:qlarge} implies
\begin{equation}\label{ineq:ck:bound}
    c_{k+1} > \left(\frac{t_k}{q_k}\right)^2 + \frac{t_k}{q_k} \cdot \frac{t_{k-1}}{q_{k}}- \frac{q_{k-1}}{q_{k}}.
\end{equation}

Applying Lemma \ref{lemma:quotient:tlarge} with $k=1$, we get
$$
b_1 +1 >c_1 (c_1-1).
$$
As $c_1(c_1-1)$ and $b_1$ are integer numbers, we get $b_1 \geq c_1 (c_1-1). $
In particular, this means that 
\begin{equation}\label{ineq:b-c}
    b_1 -c_1 \geq c_1 (c_1-2).
\end{equation}

Now we let $k=1$ in the inequality \eqref{ineq:ck:bound} and use the bound $b_1 \geq c_1 (c_1-1) $ to get
$$
c_2 > (c_1-1)^2 +1 - \frac{2}{c_1} \geq (c_1-1)^2 +1 - \frac{2}{3}.
$$
As $c_2$ is integer, we get $c_2 \geq (c_1-1)^2 +1$ and so the inequality $c_2\geq(c_1-1)^2$ is trivially satisfied for any $c_1\geq3$. Next, as $b_n\geq c_n$ for all positive integer $n$, we have $t_n\geq q_n$, and so
$$
\frac{t_k}{q_k} \cdot \frac{t_{k-1}}{q_{k}}- \frac{q_{k-1}}{q_{k}} \geq 0.
$$
Thus,
$$
 c_{k+1} > \left(\frac{t_k}{q_k}\right)^2.
$$

Applying Lemma \ref{lemma:quotient:tnqn} to the sequences $\{c_n\}_{n\geq1}$ and $\{b_n\}_{n\geq1}$ defined by the algorithm \eqref{algorithm:bn}, \eqref{algorithm:cn}, for $k\geq2$ we get
\begin{equation}\label{ineq:ck+1:continuants}
c_{k+1} > \left(1 + \frac{b_1-c_1 }{c_1 +  1/c_2} \right)^2.
\end{equation}

Substituting the bound for $b_1-c_1$ into \eqref{ineq:ck+1:continuants} and using the fact that $c_2\geq (c_1-1)^2 +1$, we get
$$
c_{k+1} >  \left(1 + \frac{c_1 (c_1-2)}{c_1 +  \frac{1}{(c_1-1)^2 +1}} \right)^2.
$$
Easy to check that we have
$$
\left(1 + \frac{c_1 (c_1-2)}{c_1 +  \frac{1}{(c_1-1)^2 +1}} \right)^2 > (c_1-1)^2-1
$$
for any $c_1\geq3$. As $c_{k+1}$ is an integer we get the desired conclusion.
\end{proof}

\begin{remark}
The condition $c_1 \geq3$ cannot be removed without making the bound \eqref{ineq:cn:growing} worse. Indeed, one can easily check that for any number $x\in\left(\frac56,\frac9{10}\right)$ one has $c_1(x)=b_1(x)=c_2(x)=2$. The case $c_1=2$ corresponds to the sum $S(2)+S(2)$, which was already covered by Corollary \ref{coroll:cusick}, so it is not a restrictive condition.
\end{remark}

\begin{remark}
   The bound \eqref{ineq:cn:growing} is not optimal, however, it is enough for our needs to get the desired statements.
\end{remark}

\begin{remark}
   Note that similarly to the inequality \eqref{ineq:ck:bound}, we can use the inequality \eqref{inequality:lemma2} from Lemma \ref{lemma:quotient:tlarge} to get a bound on $b_k$ of the form
\begin{equation}\label{ineq:bk:bound}
    b_{k}> \left( \frac{q_{k}}{t_{k-1}} \right)^2 - \frac{q_{k-1}}{t_{k-1}} \cdot \frac{q_{k}}{t_{k-1}} - 1 -\frac{t_{k-2}}{t_{k-1}}.
\end{equation}
\end{remark}

\section{Proof of the main results}\label{sec:proofs}

In this section, we extract the main results from the algorithm and statements above. We start by disproving the conjecture of Cusick.
\begin{proof}[Proof of Theorem \ref{theorem:main}]
For any $x\in\left(0, \frac{1}{k-1} \right]$ we clearly have $c_1(x)\geq k$. By Lemmas \ref{lemma:growth:bn} and \ref{lemma:growth:cn} we get that $b_n(x)\geq c_n(x) \geq k$ for any $n\in\N$. Lemmas \ref{bk_correctly_defined} and \ref{ck+1_correctly_defined} imply that the algorithm is defined correctly and inequality \eqref{eq:good:second} shows that the equation \eqref{eq:convergence} is true, i.e. that the algorithm converges to $x$.
\end{proof}

Similarly, we can show the result for the sum of two different Cantor sets, which is Theorem \ref{main:theorem2}.

\begin{proof}[Proof of Theorem \ref{main:theorem2}]

Fix integer $m,n$ , such that 
$3 \leq m < n\leq (m-1)^2$. We want to build $c\in S(m)$ and $b\in S(n)$. For any number $x\in (0,\frac{1}{m-1}]$ by algorithm we get $c_1\geq m$. Then, as in the proof of Lemma \ref{lemma:growth:cn} we get 
$$b_1 \geq c_1(c_1-1) \geq m(m-1)\geq n.$$

As $c_1\geq m \geq3$, the condition of Lemma \ref{lemma:growth:cn} is satisfied and we get $c_{n+1}\geq (c_1-1)^2 \geq (m-1)^2$ for any $n\geq1$. Also, by Lemma \ref{lemma:growth:bn} we get $b_{k+1}\geq c_{k+1},$ so we also have $b_{k+1}\geq (m-1)^2\geq n$. The fact that the algorithm is already known by the proof of Theorem \ref{theorem:main}.
\end{proof}

\begin{proof}[Proof of Theorem \ref{thm:square}]
   Assume $m\geq3$. We want to take any number $x\in(0,\frac{m+1}{m^2}]$ and decompose it as $x=c+b$, where $c\in S(m)$ and $b\in S(m^2)$.

    First, note that $\frac{m+1}{m^2} < \frac{1}{m-1}$, and thus $c_1(x)\geq m$ for any $x\in(0,\frac{m+1}{m^2}]$. As before, $b_1:=b_1(x)$ satisfies $b_1\geq c_1(c_1-1)$. Now we consider two options.

    \begin{enumerate}
        \item One has $c_1\geq m+1$. Then $b_1\geq c_1(c_1-1)\geq m(m+1)> m^2$.
        Also, by Lemma \ref{lemma:growth:cn} we get $c_{n+1}\geq(c_1-1)^2$ for any $n\geq1$. As $c_1\geq m+1$ and $b_{k+1}\geq c_{k+1}$ by Lemma \ref{lemma:growth:bn}, we get
        
$$
b_{k+1}\geq c_{k+1} \geq (c_1-1)^2 \geq m^2.
$$    
So we get the desired bound.

        \item One has $c_1=m$. Then $x-\frac{1}{c_1} \in \left(0,\frac{1}{m^2}\right)$, and thus $b_1\geq m^2$. Using \eqref{ineq:ck:bound} with $k=1$, we get
        $$
c_{2} > m^2 + m \cdot \frac{1}{m} - \frac{1}{m}=m^2+1-\frac{1}{m}.
$$
As $c_2$ is integer, we get $c_2\geq m^2+1$, and so also $b_2\geq m^2+1$. Now for $k\geq2$ by \eqref{ineq:ck+1:continuants} we have
$$
c_{k+1} > \left(1 + \frac{m^2-m }{m +  \frac{1}{m^2+1}} \right)^2 > m^2-1.
$$
As $c_{k+1}$ are integers, and $b_{k+1}\geq c_{k+1}$, we get $b_{k+1}\geq m^2$, which is the desired bound.
    \end{enumerate}

    The remaining case is $m=2$, that is we want show that $S(2)+S(4)=[0,3/4]$. Note that by Theorem \ref{main:theorem2} for $m=3,n=4$ we get
    $$
    [0,1/2]\subseteq S(3)+S(4)\subseteq S(2)+S(4), $$
    so we only need to prove the result for the interval $(1/2,3/4)$. On this interval we always have $c_1=2$, and $b_1\geq 4$. Therefore, by \eqref{ineq:ck:bound} for $k=1$ we get $c_2>4+1/2$, so $c_2\geq 5$, and thus $b_2\geq5$, too. Then by \eqref{ineq:ck+1:continuants} for $k\geq2$ we get
$$
c_{k+1} > \left(1+\frac{4}{2+1/5} \right)^2 = \frac{441}{121}>3.
$$
So $c_{k+1}\geq4$, and therefore $b_{k+1}\geq4$ for any $k\geq2$ and the proof is finished.
\end{proof}





\section{Gaps in $S(k)+S(k)$}\label{sec:gaps}

Before proving Theorem \ref{prop:gaps}, let us provide a couple of results to better understand the arrangement of the intervals $G_{k,n}$. 

First, let us introduce some notation. Let
$$
G_{k}^{odd}:=\bigcup_{n=1}^\infty G_{k,2n-1}  \quad\text{ and } \quad G_{k}^{even} := \bigcup_{m=1}^\infty G_{k,2m}.
$$

In the proofs, we will also make use of the denominator of $M_{k,n}$, which we denote by $d_{k,n}$ (or just $d_n$ when $k$ is fixed and there's no confusion), so that
$$
d_n=d_{k,n} := \langle \underbrace{k,\ldots,k}_{n \text{ times }} \rangle. 
$$
Using this notation, we see that the denominator of $m_{k,n}$ is equal to $d_n+d_{n-1}$.

Finally, by $S_k$ we denote a metallic mean of a natural number $k$, that is the unique positive real root of 
$$x^2-kx-1=0.$$

\begin{proposition}\label{prop:metall}
For any fixed $k\geq3$, for any $x\in G_{k}^{odd}$ and any $y\in G_{k}^{even}$ we have 
$$
x < 2S_k^{-1} < y.
$$
\end{proposition}

\begin{proof}
    It is well-known that the metallic mean $S_k$ has a continued fraction 
    $S_k = k+[k,k,\ldots].$ Therefore, $S_k^{-1}= [k,k,k,\ldots].$ So it follows that 
$$
2S_k^{-1} = [k,k,k,\ldots] +[k,k,k,\ldots].
$$

Because of the definition of $G_{k,n}$, to prove the claim of this proposition, it is enough to show that 
$$
2M_{k,2m} < 2S_k^{-1} < 2M_{k,2n+1}
$$
for any $n,m\geq1$. This follows immediately from the fact that for each $m\geq1$, the number $M_{k,2m}$ is an even convergent to $S_k^{-1}$, whereas for any $n$, the number $M_{k,2n+1}$ is an odd convergent to $S_k^{-1}$.
\end{proof}
Next, we show one more result about the structure of the intervals $G_{k,n}$.
\begin{proposition}\label{prop:disjoint}
  For any fixed $k\geq3$ the intervals $G_{k,n}$ are disjoint.
\end{proposition}

\begin{proof}
Fix $k\geq3$. From Proposition \ref{prop:metall} we know that 
$$
G_k^{odd} \bigcap G_k^{even} = \emptyset.
$$

Thus, it remains to show that intervals in the collection $\{ G_{k,2n-1}\}_{n=1}^\infty$ are disjoint and intervals in the collection $\{ G_{k,2m}\}_{m=1}^\infty$ are disjoint, too. Let us only prove the claim for $\{ G_{k,2n-1}\}_{n=1}^\infty$ as the other one will follow similar lines with minor changes.

In fact, we will show a stronger statement: any sequence of elements $\{x_{2n-1}\}_{n\geq1}$, such that $x_{2i-1}\in G_{k,2i-1}$ for each $i$, is strictly increasing. Note that for the other part, the proof of which we skip, the correct statement would be that any sequence of elements $\{x_{2n}\}_{n\geq1}$, such that $x_{2i}\in G_{k,2i}$ for each $i$, is strictly decreasing.

It is enough to show that 
\begin{equation}\label{ineq:ordering}
M_{k,2n-1} +m_{k,2n-1}< 2M_{k,2n}< M_{k,2n+1}+m_{k,2n+1}.
\end{equation}

The first inequality is equivalent to 
$$
\frac{1}{d_{2n}d_{2n-1}}=M_{k,2n-1} -M_{k,2n} < M_{k,2n} -m_{k,2n-1}  = \frac{k-1}{d_{2n}(d_{2n-1}+d_{2n-2})},
$$
which in turn is equivalent to
$$
(d_{2n-1}+d_{2n-2})< (k-1)d_{2n-1},
$$
which is true because $k\geq3$ and $d_{2n-2}<d_{2n-1}$.

The second inequality from \eqref{ineq:ordering} is equivalent to
$$
-\frac{1}{d_{2n}(d_{2n+1}+d_{2n})} =M_{k,2n} -m_{k,2n+1} < M_{k,2n+1} -M_{k,2n} = \frac{1}{d_{2n}d_{2n+1}},
$$
which is trivially true, as the left-hand side is negative, while the right-hand side is positive.
\end{proof}








The structure of the gaps is visualized in Figure \ref{figure1}. In green color, there are intervals/points for which we know that they belong to $S(k)+S(k)$. In red, there are intervals (gaps), that do not lie in $S(k)+S(k)$. In black, there are intervals for which we do not know whether they lie in $S(k)+S(k)$ or not. The proportions are chosen for better visibility of the structure of the gaps and they do not accurately represent the lengths of the intervals. For instance, the green interval $(0,\frac{1}{k-1})$ in reality occupies more than half of the maximal interval.

\begin{figure}\label{figure1}
\centering
\begin{tikzpicture}[scale=20]  

\def\MkOne{1/3}       
\def\MkTwo{3/10}      
\def\MkThree{10/33}   
\def\mkOne{1/4}       
\def\mkTwo{3/13}      
\def\mkThree{10/43}   

\def\GkoneA{\MkOne + \mkOne}  
\def\GkoneB{\MkTwo+\MkTwo}           

\def\GktwoA{ \MkThree+\MkThree}      
\def\GktwoB{\MkTwo + \mkTwo} 

\def\GkthreeA{\MkThree + \mkThree}  
\def\GkthreeB{2 * \MkThree}         

\def\GkfourA{2 * \MkThree}         
\def\GkfourB{\MkThree + \mkThree}  

\draw[line width=0.3mm] (0, 0) -- (2/3, 0);  

\draw[line width=0.4mm, red] (0.25, 0) -- (0.37, 0) node[midway, above] {$G_{k,1}$};
\draw[line width=0.4mm, red] (0.545, 0) -- (0.61, 0) node[midway, above] {$G_{k,2}$};

\draw[line width=0.4mm, red] (0.40, 0) -- (0.44, 0) node[midway, above] {$G_{k,3}$};

\draw[line width=0.4mm, red] (0.505, 0) -- (0.525, 0) node[midway, above] {$G_{k,4}$};

\draw[line width=0.4mm, red] (0.45, 0) -- (0.46, 0) 
node[midway, above] {$G_{k,5}$};

\draw[line width=0.4mm, red] (0.465, 0) -- (0.47, 0);

\draw[line width=0.4mm, red] (0.491, 0) -- (0.498, 0) ;

\draw[line width=0.25mm, green] (0, 0) -- (0.2, 0) ;



\filldraw [green,fill=green] (1/5, 0) circle (0.03pt);
\filldraw [black,fill=green]  (2 / 3, 0) circle (0.05pt);
\filldraw [green,fill=green]  (0, 0) circle (0.03pt);
\filldraw [black,fill=green]  (2*0.30277-1/8, 0) circle (0.1pt);

\filldraw [black,fill=green]  (0.25, 0) circle (0.07pt);
\filldraw [black,fill=green]  (0.37, 0) circle (0.07pt);
\filldraw [black,fill=green]  (0.545, 0) circle (0.07pt);
\filldraw [black,fill=green]  (0.61, 0) circle (0.05pt);
\filldraw [black,fill=green]  (0.4, 0) circle (0.05pt);
\filldraw [black,fill=green]  (0.44, 0) circle (0.05pt);
\filldraw [black,fill=green]  (0.505, 0) circle (0.04pt);
\filldraw [black,fill=green]  (0.525, 0) circle (0.04pt);
\filldraw [black,fill=green]  (0.45, 0) circle (0.04pt);
\filldraw [black,fill=green]  (0.46, 0) circle (0.04pt);
\filldraw [black,fill=green]  (0.465, 0) circle (0.03pt);
\filldraw [black,fill=green]  (0.47, 0) circle (0.03pt);
\filldraw [black,fill=green]  (0.491, 0) circle (0.03pt);
\filldraw [black,fill=green]  (0.498, 0) circle (0.03pt);

\node[below] at (0, 0) {0};
\node[below] at (2/3, 0) {$\frac{2}{k}$};
\node[below] at (1/5, 0) {$\frac{1}{k-1}$};
\node[below] at (2*0.30277-1/8, 0) {$2S_k^{-1}$};

\end{tikzpicture}
\caption{Known structure of $S(k)+S(k)$.}
\end{figure}
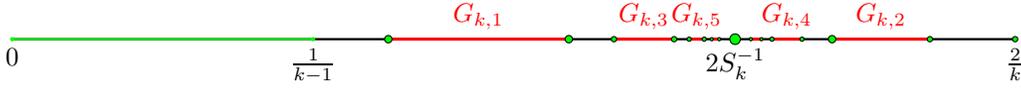

Now let us prove Theorem \ref{prop:gaps}.

\begin{proof}[Proof of Theorem \ref{prop:gaps}]
Fix some $k\geq3$ and $n\geq1$.  Consider a case of $n$ being an odd number. We show that any $x\in G_{k,n}$ is not expressible as a sum of two elements of $b,c\in S(k)$.

Assume that the statement of the theorem is not true, that is some $x\in G_{k,n}$ is expressible. We distinguish several cases of properties of $c$ and $b$, and we will come to a contradiction in all of them. 

First, let us assume that continued fraction expansions of both $c$ and $b$ have at least $n$ partial quotients. Let us show that in this case in the expression $x=c+b$, the first $n$ partial quotients of both $b$ and $c$ should be equal to $k$. Once again, assume the opposite, that there is at least one partial quotient with a value at least $k+1$. Note that we don't need to consider a case of partial quotients less than $k$ as we want to decompose $x$ into a sum of two elements of $S(k)$.

By our assumption, there exist $1\leq j \leq n$,  such that $c_i=b_i=k$ for $i=1,\ldots,j-1$, but $c_j\geq k+1$ or $b_j\geq k+1$. Because of this, $x$ is expressible as
\begin{equation}\label{eq:whatever}
x = [\underbrace{k,\ldots,k}_{j-1 \text{ times }},k+C_1,\ldots] + [\underbrace{k,\ldots,k}_{j-1 \text{ times }},k+C_2,\ldots]
\end{equation}
for some integer $C_1\geq1,C_2\geq0$.

We have two options:
\begin{itemize}
    \item If $j$ is odd, then we can show that $x\leq [\underbrace{k,\ldots,k}_{n \text{ times }}] + [\underbrace{k,\ldots,k}_{n \text{ times }},1]=M_{k,n}+m_{k,n}$.

Indeed, from \eqref{eq:whatever} by the rule of comparison of continued fractions, we can bound 
$$
x =[\underbrace{k,\ldots,k}_{j-1 \text{ times }},k+C_1,\ldots] + [\underbrace{k,\ldots,k}_{j-1 \text{ times }},k+C_2,\ldots] \leq [\underbrace{k,\ldots,k}_{j-1 \text{ times }},k+1] + [\underbrace{k,\ldots,k}_{j \text{ times }}]= m_{k,j}+M_{k,j}.
$$
By the proof of Proposition \ref{prop:disjoint}, we have $M_{k,j}+m_{k,j} \leq M_{k,n}+m_{k,n}$ when $j\leq n$ and both $j,n$ are odd.


    \item If $j$ is even, in particular it means $j\leq n-1$, as $n$ is odd, then we can show that 
    $$
    x>[\underbrace{k,\ldots,k}_{n+1 \text{ times }}] + [\underbrace{k,\ldots,k}_{n+1 \text{ times }}]=2M_{k,n+1}.
    $$

Indeed, from \eqref{eq:whatever} by the rule of comparison of continued fractions, we can bound 
$$
x =[\underbrace{k,\ldots,k}_{j-1 \text{ times }},k+C_1,\ldots] + [\underbrace{k,\ldots,k}_{j-1 \text{ times }},k+C_2,\ldots] > [\underbrace{k,\ldots,k}_{j \text{ times }}] + [\underbrace{k,\ldots,k}_{j \text{ times }},1]=M_{k,j}+m_{k,j}.
$$

By the definition of $G_{k,n}$, we have $M_{k,j}+m_{k,j} \geq 2M_{k,n+1}$ when $j,n+1$ are even.

\end{itemize}

This means that we can assume that the expression starts as
$$
x = [\underbrace{k,\ldots,k}_{n \text{ times }},\ldots] + [\underbrace{k,\ldots,k}_{n \text{ times }},\ldots].
$$
As $n+1$ is even, we get
$$
x > [\underbrace{k,\ldots,k}_{n+1 \text{ times }}] + [\underbrace{k,\ldots,k}_{n+1 \text{ times }}] = 2M_{k,n+1}.
$$
Thus, we are outside $G_{k,n}$ once again, which contradicts $x\in G_{k,n}$.

Second, we consider the case where at least one number from $c,b$ has at most $n-1$ partial quotients. Without loss of generality, assume that the number $c=[c_1,\ldots,c_m]$ has  $m\leq n-1$ partial quotients, and the number of partial quotients of $b$ is not less than the number of partial quotients of $c$. We can in fact assume that the number of partial quotients of $b$ is strictly greater than $m$, because if it is also equal to $m$, then, depending on the parity of $m$, both $b$ and $c$ would simultaneously be two even or two odd convergents to both endpoints of $G_{k,n}$, so $b+c$ would always be outside of $G_{k,n}$, leading to a contradiction.

Similarly to the previous case, we can easily check that the first $m$ partial quotients of both numbers are equal to $k$, so we have

\begin{equation}\label{eq:parity:m}
x = [\underbrace{k,\ldots,k}_{m \text{ times }}] + [\underbrace{k,\ldots,k}_{m \text{ times }},\ldots].
\end{equation}
At this point we need to analyze the parity of $m$. The general idea is as follows: if $m$ is even, then the first term in \eqref{eq:parity:m} is an even convergent, so the sum would be "small" no matter what happens with the second term after $m$ partial quotients; if the $m$ is odd, then the first term is \eqref{eq:parity:m} is an odd convergent, so the sum would be "large" no matter what happens with the second term after $m$ partial quotients. Now, let us show it rigorously.

If $m=2l$ is even, then
$$
x = [\underbrace{k,\ldots,k}_{2l \text{ times }}] + [\underbrace{k,\ldots,k}_{2l \text{ times }},\ldots] \leq [\underbrace{k,\ldots,k}_{n-1 \text{ times }}] + [\underbrace{k,\ldots,k}_{n \text{ times }}] \leq [\underbrace{k,\ldots,k}_{n \text{ times }}] + [\underbrace{k,\ldots,k}_{n \text{ times }},1] = M_{k,n}+m_{k,n}.
$$

If $m=2l+1$ is odd, then 
$$
x = [\underbrace{k,\ldots,k}_{2l+1 \text{ times }}] + [\underbrace{k,\ldots,k}_{2l+1 \text{ times }},\ldots] \geq [\underbrace{k,\ldots,k}_{n-2 \text{ times }}] + [\underbrace{k,\ldots,k}_{n-1 \text{ times }}]
$$
We want to prove that right-hand side of the previous inequality is greater or equal to  
$$ [\underbrace{k,\ldots,k}_{n+1 \text{ times }}] + [\underbrace{k,\ldots,k}_{n+1 \text{ times }}] = 2M_{k,n+1}.
$$
So we want to check that
$$
[\underbrace{k,\ldots,k}_{n-2 \text{ times }}] + [\underbrace{k,\ldots,k}_{n-1 \text{ times }}] \geq [\underbrace{k,\ldots,k}_{n+1 \text{ times }}] + [\underbrace{k,\ldots,k}_{n+1 \text{ times }}].
$$
The latter inequality is equivalent to
$$
M_{k,n-2}-M_{k,n+1}=[\underbrace{k,\ldots,k}_{n-2 \text{ times }}] -[\underbrace{k,\ldots,k}_{n+1 \text{ times }}]  \geq [\underbrace{k,\ldots,k}_{n+1 \text{ times }}] - [\underbrace{k,\ldots,k}_{n-1 \text{ times }}] =M_{k,n+1}-M_{k,n-1}.
$$
One can find those differences explicitly to check that the last inequality is true, however we will provide a simpler geometrical argument. All numbers in the above inequality are convergents to $S_k^{-1}$, so we have
$$
M_{k,n-2}-M_{k,n+1} > M_{k,n-2}- S_k^{-1} >S_k^{-1} - M_{k,n-1} >  M_{k,n+1}-M_{k,n-1},
$$
where we used standard properties of convergents, namely that each next one is closer to the number than the previous one and that odd convergents are always greater than the irrational number to which they are convergents, while even one are always smaller. So we came to a contradiction once agian.

Therefore, any number from $G_{k,n}$ cannot be expressed.

The case of even $n$ can be done similarly with minor changes.

\end{proof}


\section{Numerical data, final remarks and open problems}\label{sec:remarks}

First, we provide numerical experiments for some particular values of $x$. As it will be seen, the algorithm (which practically just uses the Gauss map on some shifted versions of the original number) seems to provide approximations of reasonable quality, which potentially may have some applications in the theory of regular diophantine approximations. 

\begin{itemize}
    \item For $x=\pi-3=[7, 15, 1, 292, 1, 1,\ldots]$, the algorithm gives decomposition
    $$
\pi-3 =   [8, 211, 73445474, 4286135421,\ldots]+[60, 58016, 1553951245, 204528884225,\ldots]
$$
with precision 
$$
\left|\pi - 3 - [8, 211, 73445474, 4286135421]-[60, 58016, 1553951245, 204528884225]\right|<10^{-42}.
$$

\item For $x=e-2=[1, 2, 1, 1, 4, 1, 1, 6, 1,\ldots]$,  the algorithm gives decomposition
    $$
e-2 = [2, 8, 47, 138, 790, 3088,\ldots] + [4, 26, 81, 349, 940, 41582,\ldots]
    $$
with precision 
$$
|e-2 - [2, 8, 47, 138, 790, 3088] - [4, 26, 81, 349, 940, 41582] | < 10^{-28}.
$$

\item For $x=2\cdot(\sqrt2-1)=2\cdot[2,2,2,\ldots]$, the algorithm gives decomposition
$$
2\cdot(\sqrt2-1) = [ 2, 51, 139299, 23380586,\ldots] + [ 3, 2143, 8527219, 38512412,\ldots]
$$
with precision
$$
| 2\cdot(\sqrt2-1) - [ 2, 51, 139299, 23380586] + [ 3, 2143, 8527219, 38512412] | < 10^{-36}.
$$

\item For $x= 34/55=[1, 1, 1, 1, 1, 1, 1, 1, 1]$, the algorithm gives decomposition

$$
\frac{34}{55} = [2,37]+[8,103].
$$
\end{itemize}

One can see that in these examples not only is each sequence of partial quotients $\{c_n\}_{n\geq1}$ and $\{b_n\}_{n\geq1}$ non-decreasing, but a 'merged' sequence $c_1,b_1,c_2,b_2,\ldots$ is non-decreasing, too. Unfortunately, in general, the last observation does not hold for all numbers. One can easy construct numbers $x$ for which $c_2(x) < b_1(x)$. For example, one can check that for all numbers $x\in(10/21,23/48)$ we have $c_1(x)=3,b_1(x)=6,c_2(x)=5$. Nevertheless, this suggests the following open problem.

\begin{problem}\label{problem1}
Are the sequences of partial quotients $\{c_n\}_{n\geq1}$ and $\{b_n\}_{n\geq1}$ defined by algorithm \eqref{algorithm:bn},\eqref{algorithm:cn} respectively non-decreasing? 
\end{problem}

Let $G$ be a set of numbers with partial quotients tending to infinity first defined by Good in \cite{MR0004878}, so that
$$
G = \{ x\in(0,1) : a_n \to\infty \text{ as } n\to\infty \}.
$$
This set has many applications for obtaining bounds of Hausdorff dimension in metrical continued fractions theory. The heuristics above suggest the following question.
\begin{problem}\label{problem2}
    Can every real number $x\in(0,1]$ be expressed as a sum of two elements $x=g_1+g_2$, where $g_1,g_2\in G$?
\end{problem}
In a stronger version of Problem \ref{problem2} one can ask for partial quotients of $g_1,g_2$ to form a strictly (eventually) increasing sequence. One can note that a positive answer to Problem \ref{problem1} would imply a positive answer to a stronger version of Problem \ref{problem2}.

A 'half' of Problem \ref{problem1} is already known by Lemma \ref{lemma:growth:bn} above, where we showed that $b_n\geq c_n$ for all $n$. Using \eqref{ineq:ck:bound} and \eqref{ineq:bk:bound} one can easily get an inequality of the form $c_{n+1}\geq \theta b_n$ for all $n\in\N$ and for some $\theta<1$ (for instance, one can show this with $\theta=1/2$).









\bibliographystyle{abbrv}
\bibliography{bibliog}

\end{document}